\documentclass[11pt,a4paper,reqno]{amsart}

\usepackage{amssymb,amsmath}
\usepackage{amsthm}
\usepackage{multicol}
\usepackage{latexsym}
\usepackage{longtable}
\usepackage[colorlinks=true,linktocpage=true,pagebackref=false, citecolor=black,linkcolor=black]{hyperref}
\usepackage{graphics}
\usepackage{euscript}
\usepackage{mathrsfs}
\usepackage[dvips]{curves}
\usepackage{tikz}
\usepackage{pgfplots}
\usetikzlibrary{arrows,decorations.pathmorphing,backgrounds,positioning,fit,matrix}
\usepackage{comment}
\usepackage{multirow}
\usepackage{xcolor}
\usepackage{caption,subcaption}
\usepackage{capt-of}

\pgfplotsset{every axis/.append style={
                    axis x line=center,    
                    axis y line=center,    
                    axis line style={->}, 
                    ytick={1,2,3,4},
                    xtick={1,2,3,4},                 
                    }}

\tikzset{>=stealth}

\def\sqbullet{\raise.2ex\hbox{\vrule width 3.5pt height 3.5pt}}

{\em}

\newcounter{substep}
\def\thesubstep{\arabic{substep}}

{\em}

\newcounter{subsubstep}
\def\thesubsubstep{\arabic{subsubstep}}

{\em}

\newtheorem{thm}{Theorem}[section]

\newtheorem{lem}[thm]{Lemma}

\theoremstyle{definition}

\theoremstyle{remark}

\newtheorem{quest}[thm]{Question}

\numberwithin{equation}{section}

 \newcommand{\R}{{\mathbb R}}
 \newcommand{\C}{{\mathbb C}}

\newcommand{\Ss}{{\EuScript S}}

\newcommand{\Qq}{{\EuScript Q}}

\newcommand{\Aa}{{\EuScript A}}

\newcommand{\Bb}{{\EuScript B}}

\newcommand{\tildebaja}{{\raise.17ex\hbox{$\scriptstyle\sim$}}}

\newcommand{\x}{{\tt x}} \newcommand{\y}{{\tt y}}

\numberwithin{equation}{section}

\begin{document}
\title{A short proof for the open quadrant problem}
\author{Jos\'e F. Fernando}
\address{Departamento de \'Algebra, Facultad de Ciencias Matem\'aticas, Universidad Complutense de Madrid, 28040 MADRID (SPAIN)}
\curraddr{Dipartimento di Matematica, Universit\`a degli Studi di Pisa, Largo Bruno Pontecorvo, 5, 56127 PISA (ITALY)}
\email{josefer@mat.ucm.es}

\author{Carlos Ueno}
\address{Dipartimento di Matematica, Universit\`a degli Studi di Pisa, Largo Bruno Pontecorvo, 5, 56127 PISA (ITALY)}
\email{jcueno@mail.dm.unipi.it}

\thanks{The first author is supported by Spanish GR MTM2011-22435, while the second is a external collaborator of this project. This article has been written during a common one year research stay of the authors at the Dipartimento di Matematica of the Universit\`a di Pisa. The authors would like to thank the department for  the very pleasant working conditions.}

\date{17/11/2014}

\begin{abstract}
In 2003 it was proved that the open quadrant $\Qq:=\{x>0,y>0\}$ of $\R^2$ is a polynomial image of $\R^2$. This result was the origin of an ulterior more systematic study of polynomial images of Euclidean spaces. In this article we provide a short proof of the previous fact that does not involve computer calculations, in contrast with the original one. The strategy here is to represent the open quadrant as the image of a polynomial map that can be expressed as the composition of three simple polynomial maps whose images can be easily understood.
\end{abstract}
\maketitle

\section{Introduction}
In the 1990 Reelle Algebraische Geometrie Seminar held in Oberwolfach Gamboa \cite{g} proposed the following problem: 
\begin{center}
\em Characterize geometrically the images of polynomial maps between Euclidean spaces. 
\end{center}
The effective representation of a subset $\Ss\subset\R^m$ as a polynomial or regular image of $\R^n$ reduces the study of certain classical problems in Real Geometry to its study in $\R^n$ with the advantage of avoiding contour conditions. Examples of these problems are Optimization or Positivstellens\"atze certificates \cite{fg2,fu2}. 

When facing the problem above, the fact of working over the field of real numbers introduces extra difficulties that are not present when working over the field of complex numbers. As a simple example, it is a basic result in the theory of one complex variable that the image of a non-constant polynomial map $f:\C\to\C$ is always equal to $\C$. However, the equivalent statement in the real setting no longer holds. The reader can easily verify that: {\it The image of a real, non-constant polynomial function is an unbounded closed interval.}

If we broaden our interest to polynomial maps $f:\R^n\to\R^m$ between Euclidean spaces the characterization of their images becomes a tougher task. By Tarski-Seidenberg's principle \cite[1.4]{bcr} the image of an either polynomial or regular map is a semialgebraic set. A subset $\Ss\subset\R^n$ is \em semialgebraic \em when it has a description by a finite boolean combination of polynomial equations and inequalities. During the last decade we have approached the problem of characterizing which (semialgebraic) subsets $\Ss\subset\R^m$ are polynomial or regular images of $\R^n$. On the one hand, we have obtained some necessary conditions that a semialgebraic set must satisfy in order to be a polynomial or regular image of $\R^n$ (see \cite{f1,fg1,fg2,fu1}). On the other hand, we have described how to obtain constructively notable families of semialgebraic sets as images of polynomial or regular maps. In particular, we have focused our attention in convex polyhedra, their interiors and their complementaries \cite{fgu1,fu2,fu3,u2}.

Even in low dimensions we have to deal with situations that at first sight look harmless, but when considered more carefully become unexpectedly hard to handle because of the lack of precise tools to determine the image of a polynomial map. A particular case is the positive answer to the famous `quadrant problem':
\begin{thm}\label{quad}
The open quadrant $\Qq:=\{x>0,y>0\}$ of $\R^2$ is a polynomial image of $\R^2$.
\end{thm}
This problem was stated in \cite{g} and solved in \cite{fg1}. The proof proposed in \cite{fg1} makes use of Sturm's algorithm applied to a high degree polynomial and the complexity of the involved calculations required computer assistance. This fact makes the reading of the proof rather disappointing, for it becomes a tedious task to verify that all the performed computations are indeed correct.  

We have always wondered whether a less technical and less demanding approach was possible. In this work we present a very short and elementary proof for the quadrant problem, which completely avoids the use of computers. Our approach is different to the one chosen in \cite{fg1}. Our strategy here is to provide a map $f:\R^2\to\R^2$ that can be expressed as the composition of three simple polynomial maps whose images are easily estimated and has the open quadrant as image.  To be more precise, we will show that $\Qq$ is the image of the polynomial map $f:=H\circ G\circ F:\R^2\to\R^2$ where
\begin{align}\label{maps}
&F:\R^2\to\R^2,\ (x,y)\mapsto((xy-1)^2+x^2,(xy-1)^2+y^2),\nonumber\\
&G:\R^2\to\R^2,\ (x,y)\mapsto(x,y(xy-2)^2+x(xy-1)^2),\\
&H:\R^2\to\R^2,\ (x,y)\mapsto(x(xy-2)^2+\tfrac{1}{2}xy^2,y)\nonumber.
\end{align}

Apart from providing a shorter and more comprehensible proof of the open quadrant problem, we have other reasons to revisit the issue. One is related to its  importance: \em The representation of the open quadrant as a polynomial image is a key step in order to construct polynomial or regular images of higher complexity\em, as is the case for the family of convex polyhedra that we mentioned before. Another one is related to the still pending question of finding an optimal polynomial map that achieves the goal. In other words:
\begin{center}
\em Which is the simplest polynomial map $f:\R^2\to\R^2$ whose image is the open quadrant?\em 
\end{center}
Here the term `simplest' is rather vague and ambiguous. In Section~\ref{sec:effect} we try to be more specific and provide questions that at present we are unable to answer.

\section{The new proof}
In order to prove Theorem \ref{quad} we need some preliminary work. As we have already announced, $\Qq$ is the image of a composition of three simple polynomial maps. We present next three auxiliary lemmas that show some properties of the images of the polynomial maps $F,G,H$ introduced in \eqref{maps}.

\begin{lem}\label{step1}
Let $\Aa:=\{xy-1\ge0\}\cap\Qq$. Then the image of
$$
F:=(F_1,F_2):\R^2\to\R^2,\ (x,y)\mapsto((xy-1)^2+x^2,(xy-1)^2+y^2)
$$
satisfies $\Aa\subset F(\R^2)\subset\Qq$.
\end{lem}
\begin{proof}
It is clear that $F_1,F_2$ are strictly positive on $\R^2$. Consequently, $F(\R^2)\subset\Qq$. To prove the first inclusion we show that if $a>0$, $b>0$ satisfy $ab-1\ge 0$, then the system of equations
\begin{equation}\label{F}
\begin{cases}
(xy-1)^2+x^2=a,\\
(xy-1)^2+y^2=b
\end{cases}
\end{equation}
has a solution $(x,y)\in\R^2$. Set $z:=xy-1$ and rewrite the system \eqref{F} in terms of the variables $\{x,z\}$. We have $y=\frac{z+1}{x}$ and \eqref{F} becomes
$$
\begin{cases}
z^2+x^2=a,\\
z^2+\frac{(z+1)^2}{x^2}=b.
\end{cases}
$$
We eliminate $x$ and deduce that $z$ must satisfy the polynomial equation
$$
P(z):=z^4-(a+b+1)z^2-2z+(ab-1)=0.
$$
Observe that $P$ is a monic polynomial of even degree such that 
$$
P(0)=ab-1\ge0\quad\text{and}\quad P(\sqrt{a})=-2\sqrt{a}-a-1<0.
$$ 
Thus, $P$ \em has a real root $z_0$ such that $0\le z_0<\sqrt{a}$\em. Set $x_0:=\sqrt{a-z_0^2}$ and $y_0:=\frac{z_0+1}{x_0}$. We have $F(x_0,y_0)=(a,b)$, so $\Aa\subset F(\R^2)$, as required.
\end{proof}

\begin{center}
\begin{figure}[t]
\begin{minipage}{0.49 \textwidth} 
\begin{center}
\begin{tikzpicture}[scale=1]
    \begin{axis}[
            xmin=-1,xmax=5,
        ymin=-1,ymax=5]
        \addplot [smooth,thick,domain=0.2:5,fill=gray!40] ({x}, {1/x});
        \addplot [thick,color=gray!40,fill=gray!40, 
                    fill opacity=1]coordinates {
            (5, 5) 
            (5, 0.2)
            (0.2, 5)  };
    \end{axis}
    \path (4.1,4.2) node {$\Aa$};
\end{tikzpicture}
\end{center}
\end{minipage}
\hfill 
\begin{minipage}{0.49 \textwidth} 
\begin{center}
\begin{tikzpicture}[scale=1]
    \begin{axis}[
            xmin=-1,xmax=5,
        ymin=-1,ymax=5]
        \addplot [smooth,thick,domain=0.2:5,fill=gray!40] ({x}, {1/x});
        \addplot [thick,color=gray!40,fill=gray!40, 
                    fill opacity=1]coordinates {
            (5, 5) 
            (5, 0.2)
            (0.2, 5)  };
             \addplot [thick,color=gray!40,fill=gray!40, 
                    fill opacity=1]coordinates {
            (5, 5) 
            (0.05, 0.05)
            (0.05, 5)  };
            
        \addplot[thick, domain=0.1:1] expression {x};
       \addplot[thin, dashed,domain=1:5] expression {x};
       \addplot[thin, dashed, domain=0.1:1] expression {1/x};
        \addplot [thick,dashed,            ]coordinates{
            (0.05, 5)
            (0.05, 0.1)  };
    \end{axis}
    \path (4.1,4.2) node {$\Bb$};
\end{tikzpicture}
\end{center}
\end{minipage}
\caption{The sets $\Aa:=\{xy-1\ge0\}\cap\Qq$ and $\Bb:=\Aa\cup\{y\geq x>0\}$}
\end{figure}
\end{center}

\begin{lem}\label{step2}
Let $\Bb:=\Aa\cup\{y\geq x>0\}$. Then the image of 
$$
G:=(G_1,G_2):\R^2\to\R^2,\ (x,y)\mapsto(x,y(xy-2)^2+x(xy-1)^2)
$$ 
satisfies $\Bb\subset G(\Aa)\subset G(\Qq)\subset\Qq$. 
\end{lem}
\begin{proof}
The inclusion $G(\Aa)\subset G(\Qq)$ is obvious. Observe that $G_1$ and $G_2$ are strictly positive on $\Qq$. Consequently, $G(\Qq)\subset\Qq$.

Next, we prove the inclusion $\Bb\subset G(\Aa)$. Notice first that we can express $\Bb$ as follows:
$$
\Bb=\bigsqcup_{x>0}(\{x\}\times[y_x,+\infty[):=\bigsqcup_{x>0}(\{x\}\times\Bb_x),
$$
where $y_x:=\min\{x,1/x\}$. For each $x>0$ consider the polynomial function in the variable $\y$
$$
\phi_x(\y):=\y(x\y-2)^2+x(x\y-1)^2=x^2\y^3+(x^3-4x)\y^2+(4-2x^2)\y+x.
$$
These polynomials have odd degree and positive leading coefficient because $x>0$. Observe also that $\phi_x(\frac{1}{x})=\frac{1}{x}$ and $\phi_x(\frac{2}{x})=x$. Consequently 
$$
\Bb_x={[y_x,+\infty[}\subset\phi_x([1/x,+\infty[).
$$ 
Therefore
$$
\Bb=\bigsqcup_{x>0}(\{x\}\times\Bb_x)\subset\bigsqcup_{x>0}\{x\}\times\phi_x({[1/x,+\infty[})=\bigsqcup_{x>0}G(\{x\}\times{[1/x,+\infty[})=G(\Aa),
$$
as required.
\end{proof}

\begin{lem}\label{step3}
The polynomial map
$$
H:=(H_1,H_2):\R^2\to\R^2,\ (x,y)\mapsto(x(xy-2)^2+\tfrac{1}{2}xy^2,y)
$$
satisfies $H(\Bb)=H(\Qq)=\Qq$.
\end{lem}
\begin{proof}
The inclusion $H(\Bb)\subset H(\Qq)$ is obvious. Observe that $H_1$ and $H_2$ are strictly positive on $\Qq$. Consequently, $H(\Qq)\subset\Qq$.

Next, we prove $\Qq\subset H(\Bb)$ and consequently we will have $\Qq\subset H(\Bb)\subset H(\Qq)\subset \Qq$, so $H(\Bb)=H(\Qq)=\Qq$. 

For each $y>0$ consider the polynomial in the variable $\x$
$$
\psi_y(\x):=\x(\x y-2)^2+\tfrac{1}{2}\x y^2=y^2\x^3-4y\x^2+(4+\tfrac{1}{2}y^2)\x.
$$
Notice that the set $\Bb$ can be expressed as
$$
\Bb=\bigsqcup_{y>0}(\Bb_y\times\{y\})\quad \text{ where }\Bb_y:=\begin{cases}
{]0,+\infty[}& \text{if $y\ge 1$},\\
{]0,y]}\cup{[1/y,+\infty[}&\text{if $0<y<1$}.
\end{cases}
$$
As $\psi_y(\x)$ has odd degree and positive leading coefficient, we have $\lim_{x\to\infty}\psi_y(x)=+\infty$. Moreover, it holds
$$
\psi_y(0)=0,\quad \psi_y(y)=y(y^2-2)^2+\tfrac{1}{2}y^3\quad\text{and}\quad\psi_y(\tfrac{2}{y})=y.
$$
For $0<y<1$ we have 
$$
\psi_y(y)=y(y^2-2)^2+\tfrac{1}{2}y^3=y((y^2-2)^2+\tfrac{1}{2}y^2)>y
$$ 
because $(y^2-2)^2+\frac{1}{2}y^2>1$ if $0<y<1$. As $\psi_y$ is strictly positive on ${]0,+\infty[}$, we deduce
$$
\psi_y(\Bb_y)=\begin{cases}
\psi({]0,+\infty[})={]0,+\infty[}&\text{ if $y\ge 1$}\\
\psi_y({]0,y]}\cup{[1/y,+\infty[})\supset{]0,\psi_y(y)]}\cup{[\psi_y(2/y),+\infty[}={]0,+\infty[}&\text{ if $0<y<1$}.
\end{cases}
$$
Consequently,
$$
\Qq=\bigsqcup_{y>0}({]0,+\infty[}\times\{y\})\subset\bigsqcup_{y>0}(\psi_y(\Bb_y)\times\{y\})=\bigsqcup_{y>0}H(\Bb_y\times\{y\})=H(\Bb),
$$
as required.
\end{proof}

Finally, Theorem \ref{quad} follows straightforwardly from the previous three Lemmas.

\begin{proof}[Proof of Theorem \em\ref{quad}]
Applying Lemmas \ref{step1}, \ref{step2} and \ref{step3} we deduce that 
$$
\Qq=H(\Bb)\subset(H\circ G)(\Aa)\subset(H\circ G\circ F)(\R^2)\subset(H\circ G)(\Qq)
\subset H(\Qq)=\Qq,
$$
that is, $(H\circ G\circ F)(\R^2)=\Qq$, as required.
\end{proof}

\section{Effectiveness of the new map}\label{sec:effect}

The problem of the open quadrant, together with its already known positive constructive answers, invites to search for alternative polynomial maps that also solve the problem and are optimal with respect to their algebraic complexity. This algebraic complexity can be understood in several ways. We briefly describe two possible approaches to this question.

\noindent 
(A) \em Optimal algebraic structure of the polynomial map\em. On a first look it is natural to wonder how our new obtained polynomial map looks like when completely expanded and how it compares to the previous known example in \cite{fg1}. We care about the total degree of the involved polynomial map (the sum of the degrees of its components) and its total number of (non-zero) monomials. We would like to find a polynomial map with the least possible total degree and the least possible number of monomials.

In {\sc Table}~\ref{tab:old} appear the components of the polynomial map $g(\x,\y):=(g_1(\x,\y),g_2(\x,\y))$ proposed in \cite{fg1}, while {\sc Table}~\ref{tab:new} shows those of our new map $f(\x,\y):=(f_1(\x,\y),f_2(\x,\y))$. Observe that the total degree of $g$ is $56$ while the total degree of $f$ is $72$. In addition the total number of monomials of $g$ is $168$ while the total number of monomials of $f$ is $350$. We wonder:

\begin{quest}\em
\em (1) \em Which is the minimum total degree for the set of polynomial maps $\R^2\to\R^2$ whose image is the open quadrant?

\em (2) \em Which is the sparsest polynomial map $\R^2\to\R^2$ whose image is the open quadrant?
\end{quest}
\vskip 0.2cm
{\captionof{table}{The old polynomial map}\vskip -0.4cm
\label{tab:old}
\begin{longtable}{p{14cm}}
{\tiny $g_1(\x,\y):=
\big(\x^{18}+2\x^{16}+\x^{14}\big)\y^{10}+\big(-14\x^{17}-30\x^{15}+4\x^{14}-18\x^{13}+6\x^{12}-2\x^{11}+2\x^{10}\big)\y^{9}+\big(87\x^{16}+202\x^{14}-44\x^{13}+143\x^{12}-72\x^{11}+34\x^{10}-30\x^{9}+7\x^{8}-2\x^{7}+\x^{6}\big)\y^{8}+\big(-316\x^{15}-804\x^{13}+208\x^{12}-662\x^{11}+378\x^{10}-226\x^{9}+192\x^{8}-66\x^{7}+26\x^{6}-12\x^{5}+2\x^{4}\big)\y^{7}+\big(743\x^{14}+2094\x^{12}-552\x^{11}+1985\x^{10}-1134\x^{9}+828\x^{8}-688\x^{7}+269\x^{6}-128\x^{5}+58\x^{4}-12\x^{3}+\x^{2}\big)\y^{6}+\big(-1182\x^{13}-3726\x^{11}+900\x^{10}-4046\x^{9}+2124\x^{8}-1922\x^{7}+1522\x^{6}-622\x^{5}+340\x^{4}-146\x^{3}+28\x^{2}-2\x\big)\y^{5}+\big(1289\x^{12}+4582\x^{10}-924\x^{9}+5702\x^{8}-2538\x^{7}+3022\x^{6}-2150\x^{5}+906\x^{4}-558\x^{3}+207\x^{2}-30\x+1\big)\y^{4}+\big(-952\x^{11}-3840\x^{9}+584\x^{8}-5504\x^{7}+1884\x^{6}-3286\x^{5}+1910\x^{4}-888\x^{3}+586\x^{2}-162\x+12\big)\y^{3}+\big(456\x^{10}+2096\x^{8}-208\x^{7}+3487\x^{6}-792\x^{5}+2408\x^{4}-978\x^{3}+621\x^{2}-372\x+55\big)\y^{2}+\big(-128\x^{9}-672\x^{7}+32\x^{6}-1308\x^{5}+144\x^{4}-1080\x^{3}+220\x^{2}-308\x+112\big)\y+\big(16\x^{8}+96\x^{6}+220\x^{4}+224\x^{2}+85\big),$}\\[0.2cm]
{\tiny $g_2(\x,\y):=\x^{16}\y^{12}+\big(-14\x^{15}-2\x^{13}+2\x^{12}\big)\y^{11}+\big(89\x^{14}+26\x^{12}-22\x^{11}+\x^{10}-2\x^{9}+\x^{8}\big)\y^{10}+\big(-338\x^{13}-152\x^{11}+108\x^{10}-12\x^{9}+20\x^{8}-8\x^{7}\big)\y^{9}+\big(849\x^{12}+524\x^{10}-308\x^{9}+64\x^{8}-88\x^{7}+28\x^{6}\big)\y^{8}+\big(-1476\x^{11}-1176\x^{9}+558\x^{8}-198\x^{7}+220\x^{6}-54\x^{5}\big)\y^{7}+\big(1808\x^{10}+1792\x^{8}-662\x^{7}+391\x^{6}-340\x^{5}+61\x^{4}\big)\y^{6}+\big(-1562\x^{9}-1878\x^{7}+514\x^{6}-512\x^{5}+332\x^{4}-40\x^{3}\big)\y^{5}+\big(944\x^{8}+1344\x^{6}-258\x^{5}+447\x^{4}-202\x^{3}+15\x^{2}\big)\y^{4}+\big(-398\x^{7}-644\x^{5}+86\x^{4}-254\x^{3}+74\x^{2}-4\x\big)\y^{3}+\big(121\x^{6}+206\x^{4}-22\x^{3}+90\x^{2}-18\x+1\big)\y^{2}+\big(-28\x^{5}-48\x^{3}+4\x^{2}-20\x+4\big)\y+\big(4\x^{4}+8\x^{2}+4\big).$}
\end{longtable}
\addtocounter{table}{-1}}

\vskip 0.2cm
{\captionof{table}{The new polynomial map}\vskip -0.4cm
\label{tab:new}
\begin{longtable}{p{14cm}}
{\tiny $f_1(\x,\y)= \big(4\x^{26}+20\x^{24}+41\x^{22}+44\x^{20}+26\x^{18}+8\x^{16}+\x^{14}\big)\y^{26}+\big(-104\x^{25}-480\x^{23}-902\x^{21}-880\x^{19}-468\x^{17}-128\x^{15}-14\x^{13}\big)\y^{25}+\big(32\x^{26}+1458\x^{24}+5839\x^{22}+9807\x^{20}+8554\x^{18}+4036\x^{16}+967\x^{14}+91\x^{12}\big)\y^{24}+\big(-768\x^{25}-13876\x^{23}-46860\x^{21}-69188\x^{19}-53264\x^{17}-22028\x^{15}-4564\x^{13}-364\x^{11}\big)\y^{23}+\big(113\x^{26}+9382\x^{24}+97367\x^{22}+274164\x^{20}+{{703621}\over{2}}\x^{18}+236532\x^{16}+84820\x^{14}+15018\x^{12}+{{2003}\over{2}}\x^{10}\big)\y^{22}+\big(-2486\x^{25}-75768\x^{23}-525594\x^{21}-1229924\x^{19}-1360121\x^{17}-791240\x^{15}-243544\x^{13}-36436\x^{11}-2007\x^9\big)\y^{21}+\big(231\x^{26}+27209\x^{24}+446529\x^{22}+2240074\x^{20}+{{8710925}\over{2}}\x^{18}+{{8244713}\over{2}}\x^{16}+2057114\x^{14}+538091\x^{12}+{{134535}\over{2}}\x^{10}+{{6051}\over{2}}\x^8\big)\y^{20}+\big(-4620\x^{25}-193928\x^{23}-2018472\x^{21}-7667498\x^{19}-12393340\x^{17}-9976824\x^{15}-4233956\x^{13}-931594\x^{11}-96204\x^9-3492\x^7\big)\y^{19}+\big(301\x^{26}+45304\x^{24}+{{1996081}\over{2}}\x^{22}+7201629\x^{20}+21316526\x^{18}+28643722\x^{16}+{{38977381}\over{2}}\x^{14}+6970385\x^{12}+1276041\x^{10}+107514\x^8+3108\x^6\big)\y^{18}+\big(-5418\x^{25}-285964\x^{23}-3907609\x^{21}-20635184\x^{19}-48472730\x^{17}-54092612\x^{15}-30891281\x^{13}-9219636\x^{11}-1387160\x^9-94004\x^7-2128\x^5\big)\y^{17}+\big(259\x^{26}+47243\x^{24}+{{2580609}\over{2}}\x^{22}+{{23978613}\over{2}}\x^{20}+47980885\x^{18}+90479081\x^{16}+{{167285973}\over{2}}\x^{14}+{{79543913}\over{2}}\x^{12}+9792158\x^{10}+1193474\x^8+63896\x^6+1106\x^4\big)\y^{16}+\big(-4144\x^{25}-262276\x^{23}-4387372\x^{21}-29332620\x^{19}-91032136\x^{17}-138703760\x^{15}-105761608\x^{13}-41463220\x^{11}-8305148\x^9-805172\x^7-33232\x^5-424\x^3\big)\y^{15}+\big(147\x^{26}+31738\x^{24}+1029734\x^{22}+11578342\x^{20}+{{115563913}\over{2}}\x^{18}+141146980\x^{16}+174192331\x^{14}+108825883\x^{12}+{{69479137}\over{2}}\x^{10}+5560611\x^8+418307\x^6+12832\x^4+{{227}\over{2}}\x^2\big)\y^{14}+\big(-2058\x^{25}-152936\x^{23}-3012828\x^{21}-24111820\x^{19}-92027457\x^{17}-178500672\x^{15}-178178038\x^{13}-90313924\x^{11}-23084717\x^9-2880652\x^7-162050\x^5-3480\x^3-19\x\big)\y^{13}+\big(53\x^{26}+13607\x^{24}+514692\x^{22}+6758603\x^{20}+{{79906193}\over{2}}\x^{18}+{{236896489}\over{2}}\x^{16}+183083134\x^{14}+146995602\x^{12}+{{119160879}\over{2}}\x^{10}+{{23809093}\over{2}}\x^8+1115393\x^6+44130\x^4+{{1187}\over{2}}\x^2+{{3}\over{2}}\big)\y^{12}+\big(-636\x^{25}-55808\x^{23}-1273116\x^{21}-11799438\x^{19}-52781280\x^{17}-122608140\x^{15}-150791284\x^{13}-96320302\x^{11}-30530204\x^9-4590824\x^7-302620\x^5-7466\x^3-48\x\big)\y^{11}+\big(11\x^{26}+3544\x^{24}+{{314743}\over{2}}\x^{22}+2376631\x^{20}+16126054\x^{18}+55398078\x^{16}+{{202163613}\over{2}}\x^{14}+98174729\x^{12}+48935769\x^{10}+11683851\x^8+{{2465761}\over{2}}\x^6+49595\x^4+534\x^2\big)\y^{10}+\big(-110\x^{25}-12028\x^{23}-320343\x^{21}-3389132\x^{19}-17228062\x^{17}-45771560\x^{15}-65325089\x^{13}-49297772\x^{11}-18508228\x^9-3090272\x^7-192077\x^5-2508\x^3+26\x\big)\y^9+\big(\x^{26}+499\x^{24}+{{54971}\over{2}}\x^{22}+{{963471}\over{2}}\x^{20}+3698783\x^{18}+14267155\x^{16}+{{58589293}\over{2}}\x^{14}+{{64546337}\over{2}}\x^{12}+18324283\x^{10}+4804470\x^8+{{907373}\over{2}}\x^6+{{7397}\over{2}}\x^4-553\x^2-3\big)\y^8+\big(-8\x^{25}-1344\x^{23}-44228\x^{21}-539160\x^{19}-3067508\x^{17}-9009888\x^{15}-14152800\x^{13}-11695504\x^{11}-4652360\x^9-672008\x^7+7308\x^5+4412\x^3+72\x\big)\y^7+\big(28\x^{24}+2366\x^{22}+50998\x^{20}+446884\x^{18}+1901034\x^{16}+4222671\x^{14}+4948005\x^{12}+{{5751153}\over{2}}\x^{10}+639605\x^8-{{52037}\over{2}}\x^6-16587\x^4-{{1249}\over{2}}\x^2-1\big)\y^6+\big(-56\x^{23}-2828\x^{21}-42092\x^{19}-269596\x^{17}-854940\x^{15}-1404926\x^{13}-1159452\x^{11}-389461\x^9+18128\x^7+30043\x^5+2488\x^3+9\x\big)\y^5+\big(70\x^{22}+2310\x^{20}+24418\x^{18}+114536\x^{16}+265718\x^{14}+306840\x^{12}+{{302491}\over{2}}\x^{10}-{{3703}\over{2}}\x^8-{{49077}\over{2}}\x^6-{{9289}\over{2}}-{{39\x^2}\over{2}}\x^4+{{3}\over{2}}\big)\y^4+\big(-56\x^{21}-1264\x^{19}-9568\x^{17}-32360\x^{15}-52408\x^{13}-37034\x^{11}-2472\x^9+10104\x^7+3718\x^5+16\x^3-20\x\big)\y^3+\big(28\x^{20}+439\x^{18}+2350\x^{16}+{{10971}\over{2}}\x^{14}+5505\x^{12}+1064\x^{10}-2161\x^8-1411\x^6-46\x^4+{{127}\over{2}}\x^2+1\big)\y^2+\big(-8\x^{19}-86\x^{17}-312\x^{15}-451\x^{13}-168\x^{11}+222\x^9+252\x^7+26\x^5-34\x^3-5\x\big)\y+\x^{18}+7\x^{16}+{{31}\over{2}}\x^{14}+{{19}\over{2}}\x^{12}-8\x^{10}-17\x^8-4\x^6+5\x^4+{{3}\over{2}}\x^2+{{3}\over{2}}$},
\\[0.2cm]
{\tiny $f_2(\x,\y)= \big(2\x^{10}+5\x^8+4\x^6+\x^4\big)\y^{10}+\big(-20\x^9-40\x^7-24\x^5-4\x^3\big)\y^9+\big(5\x^{10}+102\x^8+149\x^6+62\x^4+6\x^2\big)\y^8+\big(-40\x^9-312\x^7-316\x^5-84\x^3-4\x\big)\y^7+\big(4\x^{10}+149\x^8+600\x^6+395\x^4+58\x^2+1\big)\y^6+\big(-24\x^9-316\x^7-720\x^5-276\x^3-16\x\big)\y^5+\big(\x^{10}+62\x^8+397\x^6+504\x^4+85\x^2\big)\y^4+\big(-4\x^9-84\x^7-284\x^5-168\x^3\big)\y^3+\big(6\x^8+60\x^6+99\x^4+5\x^2-1\big)\y^2+\big(-4\x^7-20\x^5-8\x^3+6\x\big)\y+\x^6+2\x^4-2\x^2+1 $}
\end{longtable}
\addtocounter{table}{-1}}


\noindent 
(B) \em Optimal (multiplicative) complexity\em. Straight-Line Programs (SLP's) formalize step-by-step computations that do not require branching and can be applied to the evaluation of polynomials (see \cite[Chap.4]{bcs} and \cite{w}). Here we are particularly interested in evaluating the polynomial coordinates of our map in an effective way. As multiplications have a higher cost to compute than additions/subtractions, non-scalar complexity seems a reasonable approach to consider in the first place. In our particular case, expressing our map $f$ as a composition of three simpler maps helps to lower the complexity required to evaluate $f$ at a point. More precisely, if we rewrite~\eqref{maps} as
\begin{align}\label{maps2}
&F:\R^2\to\R^2,\ (x,y)\mapsto((xy-1)^2+x^2,(xy-1)^2+y^2),\nonumber\\
&G:\R^2\to\R^2,\ (x,y)\mapsto(x,y((xy)^2-4xy+4)+x((xy)^2-2xy+1)),\\
&H:\R^2\to\R^2,\ (x,y)\mapsto(xy(x\cdot(xy)-4x+4+\tfrac{1}{2}y),y)\nonumber,
\end{align}
an ocular inspection gives us an upper bound for the non-scalar complexity (working with real coefficients) of $4+4+3=11$\footnote{If we consider coefficients in $\C$, we can lower the non-scalar complexity bound  of the map $F$ by one.}. It is not so clear whether this bound can be achieved with the polynomial map proposed in \cite{fg1}. At this point, we wonder:

\begin{quest}\em 
Which is the minimum non-scalar complexity for the set of polynomial maps $\R^2\to\R^2$ whose image is the open quadrant?
\end{quest}

Of course, we can formulate diverse variants of this question if we consider other measures of complexity. In any case, regardless of the different approaches considered, the authors are convinced that more effective examples can be found for the open quadrant problem, and perhaps even shorter proofs.
\vspace{-5mm}
\section*{}
\noindent{\bf Acknowledgement.} The authors would like to thank Prof. Tom\'as Recio for helpful suggestions to improve and refine the presentation of this article.

\end{document}